\documentclass[11pt]{article}

\usepackage[a4paper,margin=1.12in]{geometry}
\usepackage{amsmath,amssymb,amsthm,mathtools}
\usepackage{enumitem}
\usepackage{hyperref}
\usepackage{microtype}

\newcommand{\PP}{\mathbb P}
\newcommand{\CC}{\mathbb C}
\newcommand{\cO}{\mathcal O}
\newcommand{\Span}{\langle}
\newcommand{\Rspan}{\rangle}
\newcommand{\CR}{\mathrm{CR}}

\newcommand{\id}{\operatorname{id}}

\theoremstyle{plain}
\newtheorem{theorem}{Theorem}
\newtheorem{proposition}[theorem]{Proposition}

\newtheorem{corollary}[theorem]{Corollary}

\theoremstyle{definition}
\newtheorem{definition}[theorem]{Definition}

\theoremstyle{remark}
\newtheorem{remark}[theorem]{Remark}
\newtheorem{warning}[theorem]{Warning}

\title{A Pascal-type construction of the Segre cubic and the Cremona--Richmond configuration}
\author{Piotr Pokora and Tomasz Szemberg}
\date{}

\begin{document}

\maketitle

\begin{abstract}
We present a Pascal-type residual construction in \(\mathbb P^4\). Starting from two quadruples of hyperplanes whose four diagonal intersection planes lie in a hyperplane, we show that the twelve residual planes lie on a cubic threefold. In the general case this cubic is the Segre cubic, and the construction recovers its fifteen planes and the associated Cremona--Richmond configuration. We also exhibit a point-line realization of this configuration in \(\mathbb P^4\) and show that it gives a \((5,3)\)-geprofi set.

\medskip 


\medskip

\end{abstract}

\section{Introduction and residual Pascal-type motivation}

Classical incidence theorems often have interesting residual interpretations.  Pascal's
theorem and the Braikenridge--Maclaurin theorem are basic examples: a special
incidence assumption forces a residual set of points to lie on a curve of
unexpectedly low degree.  Such phenomena have higher-dimensional analogues,
where points are replaced by linear subspaces and plane curves by hypersurfaces.

In this note we consider such a residual construction in the complex projective $4$-space.  Our
main case is the $k=4$ situation in \(\PP^4\), where two quadruples of
hyperplanes give sixteen intersection planes.  If four distinguished planes are
contained in a hyperplane, then the remaining twelve planes lie on a cubic
threefold.  We show that, in the general case, this cubic is the Segre cubic.
This connects the residual construction with the fifteen planes on the Segre
cubic, the Cremona--Richmond configuration, and a natural \((5,3)\)-geprofi set.

The residual point of view behind Pascal and Braikenridge--Maclaurin type
theorems is simple.  If two \(k\)-tuples of hypersurfaces of the same degree cut
out a complete intersection, and if a part of the complete intersection lies on
a hypersurface of small degree, then the residual part lies on a hypersurface of
complementary degree.  In the plane situation this recovers the usual
Braikenridge--Maclaurin statement; see \cite{Traves2013}.  In \(\PP^3\), one
starts with two \(k\)-tuples of planes
\[
F_1,\ldots,F_k,\qquad G_1,\ldots,G_k
\]
and the \(k^2\) lines
\[
\ell_{ij}=F_i\cap G_j.
\]
If \(k\) of these lines are coplanar, then the residual \(k^2-k\) lines lie on a surface
of degree \(k-1\).  For \(k=3\), this surface is a quadric, and one obtains a spatial
Pascal-type statement involving the two rulings of a smooth quadric surface; see
Le Van~\cite{LeVan}.

The case studied here is the next natural step: the \(k=4\) construction in
\(\PP^4\).  Its residual hypersurface is a cubic threefold, and the main point
is that this cubic is the Segre cubic.

We first record the elementary residual construction in \(\PP^4\) in a general form,
since the case \(k=4\) used below will be obtained from it by specialization.
So let
\[
F_1,\ldots,F_k,\qquad G_1,\ldots,G_k
\]
be two \(k\)-tuples of hyperplanes in \(\PP^4\), with equations
\[
F_i=(f_i=0),\qquad G_j=(g_j=0).
\]
For \(1\leq i,j\leq k\), put
\[
\Pi_{ij}=F_i\cap G_j.
\]
For a general configuration, each \(\Pi_{ij}\) is a plane.

\begin{proposition}[Explicit residual hypersurface]
\label{prop:residual-hypersurface}
Let \(\sigma\in S_k\) be a permutation.  Suppose that the \(k\) planes
\[
\Pi_{i,\sigma(i)},\qquad i=1,\ldots,k,
\]
are contained in a hyperplane
\[
H=(h=0).
\]
Assume moreover that \(H\neq F_i\) and \(H\neq G_{\sigma(i)}\) for every \(i\).  Then,
after rescaling the linear forms \(f_i\) and \(g_{\sigma(i)}\), one may write
\[
h=f_i+g_{\sigma(i)}
\qquad\text{for all }i=1,\ldots,k.
\]
With this normalization, the form
\[
\prod_{i=1}^k f_i-(-1)^k\prod_{j=1}^k g_j
\]
is divisible by \(h\).  Hence
\[
X_\sigma:\quad
\frac{
\prod_{i=1}^k f_i-(-1)^k\prod_{j=1}^k g_j
}{h}=0
\]
is a hypersurface of degree \(k-1\) in \(\PP^4\), and it contains all residual planes
\[
\Pi_{ij}=F_i\cap G_j,\qquad j\neq \sigma(i).
\]
\end{proposition}

\begin{proof}
Since
\[
\Pi_{i,\sigma(i)}=F_i\cap G_{\sigma(i)}\subset H,
\]
the linear form \(h\) belongs to the pencil generated by \(f_i\) and
\(g_{\sigma(i)}\).  Thus, after rescaling, we may assume that
\[
h=f_i+g_{\sigma(i)}.
\]
Modulo \(h\), we have
\[
g_{\sigma(i)}\equiv -f_i.
\]
Since \(\sigma\) is a permutation,
\[
\prod_{j=1}^k g_j
=
\prod_{i=1}^k g_{\sigma(i)}
\equiv
(-1)^k\prod_{i=1}^k f_i
\pmod h.
\]
This proves the divisibility by \(h\).

Let us take now \(j\neq \sigma(i)\).  On the plane
\[
\Pi_{ij}=F_i\cap G_j
\]
one has \(f_i=0\) and \(g_j=0\), and hence both products
\[
\prod_{r=1}^k f_r,\qquad \prod_{s=1}^k g_s
\]
vanish on \(\Pi_{ij}\).  Therefore the numerator
\[
\prod_{r=1}^k f_r-(-1)^k\prod_{s=1}^k g_s
\]
vanishes on \(\Pi_{ij}\).  Since \(\Pi_{ij}\) is not one of the diagonal planes contained
in \(H\), the form \(h\) does not vanish identically on \(\Pi_{ij}\) for a non-degenerate
configuration.  Hence the quotient defining \(X_\sigma\) vanishes on a dense open subset
of \(\Pi_{ij}\), and therefore on the whole plane.
\end{proof}

\section{The case \(k=4\): the residual cubic}

We take
\[
k=4,\qquad \sigma=\id.
\]
Thus we assume that the four diagonal planes
\[
\Pi_{11},\Pi_{22},\Pi_{33},\Pi_{44}
\]
are contained in a hyperplane
\[
H=(h=0).
\]
After rescaling,
\[
h=f_i+g_i,\qquad i=1,\ldots,4,
\]
and hence
\[
g_i=h-f_i.
\]

\begin{theorem}[The \(k=4\) residual cubic in \(\PP^4\)]
\label{thm:k4-cubic}
The twelve residual planes
\[
\Pi_{ij}=F_i\cap G_j,\qquad i\neq j,
\]
are contained in the cubic threefold
\[
X:\quad
\frac{f_1f_2f_3f_4-g_1g_2g_3g_4}{h}=0.
\]
Equivalently, writing \(e_r=e_r(f_1,f_2,f_3,f_4)\) for the $r$-th elementary symmetric polynomial, the equation of \(X\) is
\[
e_3-h e_2+h^2 e_1-h^3=0.
\]
More explicitly, we have
\[
\begin{aligned}
X:\quad&
f_1f_2f_3+f_1f_2f_4+f_1f_3f_4+f_2f_3f_4\\
&-h(f_1f_2+f_1f_3+f_1f_4+f_2f_3+f_2f_4+f_3f_4)\\
&+h^2(f_1+f_2+f_3+f_4)-h^3=0.
\end{aligned}
\]
\end{theorem}

\begin{proof}
This is Proposition~\ref{prop:residual-hypersurface} for \(k=4\).  Since
\[
g_i=h-f_i,
\]
we have
\[
\begin{aligned}
g_1g_2g_3g_4
&=(h-f_1)(h-f_2)(h-f_3)(h-f_4)=h^4-h^3e_1+h^2e_2-he_3+e_4.
\end{aligned}
\]
Therefore
\[
f_1f_2f_3f_4-g_1g_2g_3g_4
=
-h^4+h^3e_1-h^2e_2+he_3.
\]
Dividing by \(h\) gives us
\[
e_3-h e_2+h^2e_1-h^3.
\]
\end{proof}

\subsection{Three additional planes}

The twelve residual planes are accompanied by three more natural planes.  They are indexed
by the three partitions of \(\{1,2,3,4\}\) into two pairs, namely
\[
M_{12|34}:\quad f_1+f_2=h,\qquad f_3+f_4=h,
\]
\[
M_{13|24}:\quad f_1+f_3=h,\qquad f_2+f_4=h,
\]
and
\[
M_{14|23}:\quad f_1+f_4=h,\qquad f_2+f_3=h.
\]

\begin{proposition}
\label{prop:extra-planes}
The planes
\[
M_{12|34},\qquad M_{13|24},\qquad M_{14|23}
\]
are contained in \(X\).  Thus the construction produces fifteen distinguished planes:
\[
\Pi_{ij}\quad (i\neq j),
\qquad
M_{12|34},\ M_{13|24},\ M_{14|23}.
\]
\end{proposition}

\begin{proof}
On \(M_{12|34}\) one has
\[
f_1+f_2=h,\qquad f_3+f_4=h.
\]
Thus
\[
g_1=h-f_1=f_2,\qquad g_2=h-f_2=f_1,
\]
and similarly
\[
g_3=f_4,\qquad g_4=f_3.
\]
Therefore
\[
g_1g_2g_3g_4=f_1f_2f_3f_4.
\]
So the defining numerator of \(X\) vanishes identically on \(M_{12|34}\).  The other two
cases are identical.
\end{proof}

\subsection{Identification with the Segre cubic}

Assume now that
\[
h,f_1,f_2,f_3,f_4
\]
are linearly independent.  Then they form a coordinate system on \(\PP^4\).

\begin{proposition}
\label{prop:ten-nodes}
If $h,f_1,f_2,f_3,f_4$ are linearly independent, then the cubic \(X\) has ten ordinary double points.  They are
the six points
\[
p_I:\quad f_i=h\ \text{for } i\in I,\qquad f_j=0\ \text{for }j\notin I,
\qquad |I|=2,
\]
and the four points
\[
q_i:\quad h=0,\qquad f_j=0\ \text{for } j\neq i,
\qquad i=1,\ldots,4.
\]
\end{proposition}

\begin{proof}
Let
\[
\Phi=e_3-he_2+h^2e_1-h^3.
\]
The partial derivatives are
\[
\frac{\partial \Phi}{\partial f_i}
=
e_2(f_1,\ldots,\widehat f_i,\ldots,f_4)
-h\,e_1(f_1,\ldots,\widehat f_i,\ldots,f_4)+h^2
\]
and
\[
\frac{\partial \Phi}{\partial h}
=
-e_2+2he_1-3h^2 .
\]
We first determine the singular locus. If \(h\neq 0\), we normalize \(h=1\). Put
\[
s=f_1+\cdots+f_4,\qquad q=e_2(f_1,\ldots,f_4).
\]
The equations \(\partial \Phi/\partial f_i=0\) may be written as
\[
f_i^2+(1-s)f_i+q-s+1=0,\qquad i=1,\ldots,4.
\]
Hence the four numbers \(f_1,\ldots,f_4\) take at most two distinct values. Together with
\[
\frac{\partial \Phi}{\partial h}=0,
\qquad\text{i.e.}\qquad q=2s-3,
\]
this gives, by a short case analysis on the multiplicities of the two values, precisely the
six solutions for which two of the \(f_i\)'s are equal to \(1\) and the remaining two are
equal to \(0\). These are the points \(p_I\) with \(|I|=2\).

If \(h=0\), then \(\partial\Phi/\partial h=0\) gives \(e_2=0\). Moreover
\[
\frac{\partial\Phi}{\partial f_i}=e_2(f_1,\ldots,\widehat f_i,\ldots,f_4)=0.
\]
Since
\[
e_2(f_1,\ldots,\widehat f_i,\ldots,f_4)
=e_2-f_i\sum_{j\neq i}f_j,
\]
we get
\[
f_i(s-f_i)=0
\]
for every \(i\), where \(s=f_1+\cdots+f_4\). Thus every \(f_i\) is either \(0\) or \(s\).
Since not all \(f_i\)'s vanish, exactly one of them is non-zero. This gives the four points
\(q_i\).

It remains to check the analytic type. By symmetry, it is enough to consider \(p_{\{1,2\}}\)
and \(q_1\). Near \(p_{\{1,2\}}\), in the affine chart \(h=1\), we can write
\[
f_1=1+u_1,\quad f_2=1+u_2,\quad f_3=u_3,\quad f_4=u_4.
\]
Then the quadratic part of \(\Phi\) is
\[
-u_1u_2+u_3u_4,
\]
which is non-degenerate. Near \(q_1\), in the affine chart \(f_1=1\), write
\[
h=v_0,\quad f_2=v_2,\quad f_3=v_3,\quad f_4=v_4.
\]
The quadratic part is
\[
v_0^2-v_0(v_2+v_3+v_4)+v_2v_3+v_2v_4+v_3v_4.
\]
Its Hessian matrix has non-zero determinant, hence it is non-degenerate. Therefore all ten
singularities are ordinary double points.
\end{proof}
\begin{theorem}
\label{thm:segre-cubic}
For a general residual configuration described as above, the cubic threefold \(X\subset\PP^4\) is
projectively equivalent to the Segre cubic primal
\[
\mathfrak S_3=
\left\{
x_0+\cdots+x_5=0,\quad x_0^3+\cdots+x_5^3=0
\right\}
\subset \PP^5.
\]
Consequently, the fifteen planes from Proposition~\ref{prop:extra-planes} are precisely
the fifteen planes on the Segre cubic.
\end{theorem}

\begin{proof}
The Segre cubic primal is the unique cubic threefold with ten ordinary double points, up to
projective equivalence, see Dolgachev~\cite{DolgachevSegre}.  By
Proposition~\ref{prop:ten-nodes}, the residual cubic \(X\) has ten ordinary double points
in the general case.  Hence \(X\) is projectively equivalent to the Segre cubic.  Since the
Segre cubic contains exactly fifteen planes, the fifteen distinct planes constructed above
are the full set of planes on \(X\).
\end{proof}

\subsection{The Segre cubic and the point-plane Cremona--Richmond configuration}

In the standard model
\[
\mathfrak S_3=
\left\{
x_0+\cdots+x_5=0,\quad x_0^3+\cdots+x_5^3=0
\right\}
\subset \PP^5,
\]
the fifteen planes are indexed by partitions of the six symbols into three unordered pairs:
\[
(ab\mid cd\mid ef).
\]
The corresponding plane is
\[
\Pi_{ab|cd|ef}:\quad
x_a+x_b=x_c+x_d=x_e+x_f=0.
\]
There are
\[
\frac{6!}{(2!)^3\,3!}=15
\]
such planes.

Moreover, there are also fifteen distinguished points on \(\mathfrak S_3\), namely
\[
r_{ab}=[e_a-e_b],\qquad 0\leq a<b\leq 5.
\]
In other words, \(r_{ab}\) is the point with four coordinates equal to zero, one coordinate
equal to \(1\), and one coordinate equal to \(-1\).  These points lie on the Segre cubic
because
\[
\sum x_i=0,\qquad \sum x_i^3=1+(-1)=0.
\]
The plane \(\Pi_{ab|cd|ef}\) contains exactly the three points
\[
r_{ab},\qquad r_{cd},\qquad r_{ef}.
\]
Thus the fifteen Segre planes and the fifteen points \(r_{ab}\) form a point-plane
\((15_3,15_3)\)-configuration.  This is often called the Cremona--Richmond configuration
of planes on the Segre cubic; see Manivel~\cite{Manivel}.

\begin{remark}
The terminology of duads, synthemes and totals used below goes back to Sylvester.
In his paper \cite{Sylvester1844}, he introduced the word ``syntheme''
for an aggregate of combinations in which each monad of the given system
occurs exactly once. In the case of six monads, a duad syntheme is precisely
a partition of the six symbols into three unordered pairs, and a total is
a collection of five such synthemes containing each duad exactly once.
\end{remark}

\begin{warning}
The points \(r_{ab}=[e_a-e_b]\) should not be confused with the point-line realization
of the Cremona--Richmond configuration used in the geprofi discussion below.  The points
\(r_{ab}\) lie on the Segre cubic and are incident with the fifteen Segre planes.  The
point-line realization below uses a different set of fifteen points in \(\PP^4\), naturally
identified with duads, whose collinear triples are synthemes.
\end{warning}

\section{The point-line Cremona--Richmond configuration in \(\PP^4\)}

Let
\[
V=\left\{(x_1,\ldots,x_6)\in \CC^6\mid x_1+\cdots+x_6=0\right\}.
\]
We work in
\[
\PP(V)\simeq \PP^4.
\]
For each duad \(\{i,j\}\subset\{1,\ldots,6\}\), define a point
\[
z_{ij}=[v_{ij}]\in \PP(V),
\]
where
\[
v_{ij}=e_i+e_j-\frac13(e_1+\cdots+e_6).
\]
Equivalently, after multiplying by \(3\), \(z_{ij}\) is represented by the vector
\[
(a_1,\ldots,a_6),
\]
where
\[
a_i=a_j=2,\qquad a_k=-1\quad\text{for }k\notin\{i,j\}.
\]
Dropping the sixth coordinate gives the following \(\PP^4\)-coordinates:
\[
\begin{array}{c|l}
z_{12} & [2:2:-1:-1:-1]\\
z_{13} & [2:-1:2:-1:-1]\\
z_{14} & [2:-1:-1:2:-1]\\
z_{15} & [2:-1:-1:-1:2]\\
z_{16} & [2:-1:-1:-1:-1]\\
z_{23} & [-1:2:2:-1:-1]\\
z_{24} & [-1:2:-1:2:-1]\\
z_{25} & [-1:2:-1:-1:2]\\
z_{26} & [-1:2:-1:-1:-1]\\
z_{34} & [-1:-1:2:2:-1]\\
z_{35} & [-1:-1:2:-1:2]\\
z_{36} & [-1:-1:2:-1:-1]\\
z_{45} & [-1:-1:-1:2:2]\\
z_{46} & [-1:-1:-1:2:-1]\\
z_{56} & [-1:-1:-1:-1:2]
\end{array}.
\]
Let us recall that a \textit{syntheme} is a partition of \(\{1,\ldots,6\}\) into three unordered pairs,
\[
(ab\mid cd\mid ef).
\]
To such a syntheme we associate the line
\[
L_{ab|cd|ef}=\Span z_{ab},z_{cd},z_{ef}\Rspan.
\]

\begin{proposition}[The \(15_3\) Cremona--Richmond configuration]
The three points
\[
z_{ab},\qquad z_{cd},\qquad z_{ef}
\]
are collinear if and only if
\[
(ab\mid cd\mid ef)
\]
is a syntheme.  The fifteen points \(z_{ij}\) and the fifteen lines
\(L_{ab|cd|ef}\) form a \((15_3,15_3)\)-configuration, which means that each line contains three
points and each point lies on three lines.
\end{proposition}

\begin{proof}
If \((ab\mid cd\mid ef)\) is a syntheme, then
\[
v_{ab}+v_{cd}+v_{ef}=0.
\]
Hence the three points \(z_{ab},z_{cd},z_{ef}\) are collinear.

Conversely, suppose that three distinct points
\[
z_{ab},\quad z_{cd},\quad z_{ef}
\]
are collinear. Then the corresponding vectors are linearly dependent in
\(V\). Thus there exist scalars \(\alpha,\beta,\gamma\), not all zero,
such that
\[
\alpha v_{ab}+\beta v_{cd}+\gamma v_{ef}=0.
\]
Write
\[
A=\{a,b\},\qquad B=\{c,d\},\qquad C=\{e,f\}.
\]
Using
\[
v_A=\sum_{i\in A}e_i-\frac13(e_1+\cdots+e_6),
\]
the \(r\)-th coordinate of the above relation gives
\[
\alpha\mathbf 1_{r\in A}
+\beta\mathbf 1_{r\in B}
+\gamma\mathbf 1_{r\in C}
=
\frac{\alpha+\beta+\gamma}{3}
\qquad
\text{for every }r=1,\ldots,6.
\]
In other words, the weighted incidence degree of each symbol is constant.

If the three duads \(A,B,C\) do not form a partition of
\(\{1,\ldots,6\}\), then some symbol does not occur in any of them. Hence
the common value above is zero. Therefore the weighted incidence degree
of every symbol is zero. Looking at the graph on \(\{1,\ldots,6\}\) with
edges \(A,B,C\), this means that every vertex has weighted degree zero.
This forces all edge weights \(\alpha,\beta,\gamma\) to be zero: this is
immediate for an edge with a vertex of degree one, while in the only
remaining case the three edges form a triangle, and the equations
\[
\alpha+\beta=0,\qquad
\alpha+\gamma=0,\qquad
\beta+\gamma=0
\]
again give \(\alpha=\beta=\gamma=0\).
This contradicts the choice of \(\alpha,\beta,\gamma\).

Thus \(A,B,C\) must cover all six symbols. Since they are three duads, this
means that they are pairwise disjoint. Hence
\[
(ab\mid cd\mid ef)
\]
is a syntheme. Therefore the only collinear triples are precisely the
triples corresponding to synthemes.
\end{proof}
\begin{corollary}
The fifteen collinear triples are:
\[
\begin{array}{lll}
(z_{12},z_{34},z_{56}),&
(z_{12},z_{35},z_{46}),&
(z_{12},z_{36},z_{45}),\\
(z_{13},z_{24},z_{56}),&
(z_{13},z_{25},z_{46}),&
(z_{13},z_{26},z_{45}),\\
(z_{14},z_{23},z_{56}),&
(z_{14},z_{25},z_{36}),&
(z_{14},z_{26},z_{35}),\\
(z_{15},z_{23},z_{46}),&
(z_{15},z_{24},z_{36}),&
(z_{15},z_{26},z_{34}),\\
(z_{16},z_{23},z_{45}),&
(z_{16},z_{24},z_{35}),&
(z_{16},z_{25},z_{34}).
\end{array}
\]
\end{corollary}

\subsection{Five lines covering the fifteen points}

A \emph{total} is a set of five synthemes containing every duad exactly once.  Equivalently, it is a
decomposition of the fifteen points into five collinear triples.

One such total is
\[
\begin{aligned}
T_1&=(12\mid 34\mid 56),\\
T_2&=(13\mid 25\mid 46),\\
T_3&=(14\mid 26\mid 35),\\
T_4&=(15\mid 24\mid 36),\\
T_5&=(16\mid 23\mid 45).
\end{aligned}
\]
Thus the five lines
\[
\begin{aligned}
L_1&=\Span z_{12},z_{34},z_{56}\Rspan,\\
L_2&=\Span z_{13},z_{25},z_{46}\Rspan,\\
L_3&=\Span z_{14},z_{26},z_{35}\Rspan,\\
L_4&=\Span z_{15},z_{24},z_{36}\Rspan,\\
L_5&=\Span z_{16},z_{23},z_{45}\Rspan
\end{aligned}
\]
cover all fifteen points.  Put
\[
C_T=L_1\cup L_2\cup L_3\cup L_4\cup L_5.
\]
Then \(C_T\subset\PP^4\) is a reducible curve of degree \(5\), and
\[
Z_{\CR}:=\{z_{ij}\mid 1\leq i<j\leq 6\}\subset C_T.
\]
For completeness, we count all such coverings.
\begin{proposition}
There are exactly six coverings of the point-line Cremona--Richmond set
$Z_{CR}$ by five trisecant lines.
\end{proposition}

\begin{proof}
In the duad--syntheme model, the fifteen points of $Z_{CR}$ are indexed
by the fifteen duads of $\{1,\ldots,6\}$, and the trisecant lines are
indexed by synthemes, that is, by partitions
\[
(ab\mid cd\mid ef)
\]
of $\{1,\ldots,6\}$ into three unordered pairs. Thus a covering of
$Z_{CR}$ by five trisecant lines is the same thing as a partition of the
fifteen duads into five synthemes. Equivalently, it is a one-factorization
of the complete graph $K_6$.

We count such factorizations. Fix the syntheme
\[
(12\mid 34\mid 56).
\]
A total containing it must cover the remaining twelve duads by four
synthemes. Consider the duad $13$. Since $56$ has already been used, the
syntheme containing $13$ is either
\[
(13\mid 25\mid 46)
\quad \mbox{ or }\quad (13\mid 26\mid 45).
\]
In each of the two cases the remaining synthemes are forced. Hence every
syntheme belongs to exactly two totals.

Since there are fifteen synthemes and each total contains five of them,
the number $N$ of totals satisfies
\[
5N=15\cdot 2.
\]
Therefore $N=6$.
\end{proof}
\subsection{The \((5,3)\)-geprofi property}

We now show a somewhat surprising connection between the above considerations and the concept of geprofi sets introduced by Chiantini, Farnik, Favacchio, Harbourne, Migliore,
Szemberg and Szpond~\cite{Geprofi}. Let us recall first the relevant definition. Note that the curves and surfaces appearing in the definition may be reducible.

\begin{definition}
A finite set \(Z\subset\PP^4\) of \(bd\) points is called a \((b,d)\)-geprofi set if, for
a general projection
\[
\pi_P:\PP^4\dashrightarrow\PP^3
\]
from a point \(P\in\PP^4\), the image \(\pi_P(Z)\subset\PP^3\) is a full intersection
of a curve of degree \(b\) and a surface of degree \(d\).
\end{definition}

\begin{theorem}
\label{thm:CR-geprofi}
The point-line Cremona--Richmond set
\[
Z_{\CR}=\{z_{ij}\mid 1\leq i<j\leq 6\}\subset\PP^4
\]
is a \((5,3)\)-geprofi set.
\end{theorem}

\begin{proof}
Choose the total \(T\) from the previous section, and let
\[
C_T=L_1\cup\cdots\cup L_5.
\]
This is a degree \(5\) curve containing \(Z_{\CR}\).

We choose \(P\) outside the finite union of the secant varieties of the lines \(L_i\) and
outside the spans \(\langle L_i,L_j\rangle\), \(i\neq j\). Then the projected lines
\(\overline L_1,\ldots,\overline L_5\) are pairwise skew, and the projection is an isomorphism
on each \(L_i\).

\noindent
Let
\[
\pi_P:\PP^4\dashrightarrow\PP^3
\]
be the projection from $P$, and put
\[
\Gamma=\pi_P(Z_{\CR}),\qquad \overline C_T=\pi_P(C_T).
\]
By the choice of \(P\), the projection is defined and injective on \(Z_{\CR}\). Since it maps each \(L_i\) isomorphically onto a line \(\overline L_i\subset\PP^3\), the union of these lines
\[
\overline C_T=\overline L_1\cup\cdots\cup \overline L_5
\]
is a degree \(5\) curve in \(\PP^3\), and each \(\overline L_i\) contains exactly three
points of \(\Gamma\).

Since
\[
h^0(\cO_{\PP^3}(3))=20
\]
and \(|\Gamma|=15\), the vector space \(I_\Gamma(3)\) of cubic surfaces through \(\Gamma\)
has dimension at least \(5\).  For a fixed line \(\overline L_i\), the condition that a
cubic surface containing the three points \(\Gamma\cap \overline L_i\) should contain the
whole line \(\overline L_i\) is one additional linear condition.  For a general projection
this additional condition is not automatic.  Hence a general cubic
\[
S\in |I_\Gamma(3)|
\]
contains none of the five lines \(\overline L_i\).

For such a cubic surface \(S\), the restriction \(S|_{\overline L_i}\) is a cubic divisor
on \(\overline L_i\simeq\PP^1\).  It contains the three points
\[
\Gamma\cap \overline L_i,
\]
and, since \(S\) does not contain \(\overline L_i\), it has no other intersection with
\(\overline L_i\).  Therefore
\[
S\cap \overline L_i=\Gamma\cap \overline L_i
\]
scheme-theoretically for a general choice of \(S\).  Summing over the five components, we get
\[
S\cap \overline C_T=\Gamma.
\]
Thus \(\Gamma\) is the full intersection of the degree \(5\) curve \(\overline C_T\) and
the cubic surface \(S\).  Hence \(Z_{\CR}\) is \((5,3)\)-geprofi.
\end{proof}
In the context of full and complete intersections, it is worth to mention the following.
\begin{remark}
The fifteen Segre planes are a complete intersection. In the standard model
\[
\mathbb P(V)=\left\{\sum_{i=1}^6 x_i=0\right\}\simeq\PP^4,
\]
their union is cut out by the cubic and quintic equations
\[
\sum_{i=1}^6 x_i^3=0,\qquad
\sum_{i=1}^6 x_i^5=0.
\]
Indeed, on each plane
\[
\Pi_{ab\mid cd\mid ef}:\quad
x_a+x_b=x_c+x_d=x_e+x_f=0,
\]
all odd power sums vanish. Conversely, if \(p_1=p_3=p_5=0\), then Newton
identities give \(e_1=e_3=e_5=0\). Hence
\[
\prod_{i=1}^6(t-x_i)=t^6+e_2t^4+e_4t^2+e_6
\]
is an even polynomial, so the coordinates \(x_1,\ldots,x_6\) can be paired
as opposites. Thus the point lies on one of the fifteen Segre planes.
Consequently, the union of the fifteen planes is a set-theoretic complete intersection
\[
\left(\sum x_i^3=0\right)\cap \left(\sum x_i^5=0\right)
\subset \PP(V).
\]
Since the complete intersection has degree \(3\cdot 5=15\), and the reduced union of the
fifteen planes also has degree \(15\), the above set-theoretic equality implies the claimed
scheme-theoretic complete intersection; equivalently, one checks that the two equations
are generically transverse along each plane.
\end{remark}

\section*{Acknowledgments}
We would also like to thank \textit{LOT Polish Airlines} for being occasionally delayed, thereby granting us a little extra time to continue our mathematical conversations.

\bigskip
\noindent
Piotr Pokora, Tomasz Szemberg\\
Department of Mathematics,\\
University of the National Education Commission Krakow,\\
Podchor\c a\.zych 2,
PL-30-084 Krak\'ow, Poland. \\
\nopagebreak
\textit{E-mail address:} \\
\texttt{piotr.pokora@uken.krakow.pl}\\
\texttt{tomasz.szemberg@gmail.com}\\

\end{document}